\documentclass[12pt]{article} 
\usepackage{amsthm,amssymb,amsmath}
\usepackage{enumerate}
\usepackage{graphicx}
\usepackage{float}
\usepackage{multicol}
\usepackage[utf8]{inputenc}

\usepackage{mathpazo}
\usepackage{subfig}
\usepackage{tikz}
\usepackage{color}
\definecolor{darkblue}{rgb}{0,0,.4}
\usepackage[bookmarks,
colorlinks=true,
linkcolor=darkblue,
citecolor=darkblue,
urlcolor=darkblue]{hyperref}

\theoremstyle{plain}
\newtheorem{theorem}{Theorem}
\newtheorem{corollary}[theorem]{Corollary}

\newtheorem{proposition}[theorem]{Proposition}

\theoremstyle{definition}

\newcommand{\M}{\mathcal{M}}
\newcommand{\fp}{{\rm fp}}

\newcommand{\exc}{{\rm exc}}
\newcommand{\inv}{{\rm inv}}

\newcommand{\crs}{{\rm cr}}
\newcommand{\nes}{{\rm nes}}

\newcommand{\sh}{{\rm sh}}
\newcommand{\hor}{{\rm hor}}

\newcommand{\up}{{\rm up}}
\newcommand{\down}{{\rm down}}
\newcommand{\area}{{\rm area}}

\numberwithin{equation}{section} 
\numberwithin{theorem}{section}

\usepackage{vmargin}
\setmarginsrb{2.5cm}{1.5cm}{1.5cm}{1.5cm}{1.5cm}{0.5cm}{1cm}{0.5cm}
\usepackage{fancyhdr}
\title{Crossings and nestings  over some Motzkin objects and $q$-Motzkin numbers}

\author{
Sandrataniaina R. Andriantsoa \qquad Paul M. Rakotomamonjy\\
\small Department of Mathematics and Computer Science\\[-0.8ex]
\small Sciences and Technology, PB 906 Antananarivo 101\\[-0.8ex]
\small Madagascar\\
\small\tt \{andrian.2sandra, rpaulmazoto\}@gmail.com
}
\date{}

\begin{document} % title and abstract
	\maketitle
	\begin{abstract} 		
		We examine the enumeration of certain Motzkin objects according to the numbers of crossings and nestings. With respect to continued fractions, we compute and express the distributions of the statistics of the numbers of crossings and nestings over three sets, namely the set of $4321$-avoiding involutions, the set of $3412$-avoiding involutions, and the set of $(321,3\bar{1}42)$-avoiding permutations. To get our results, we exploit the bijection of Biane restricted to the sets of $4321$- and $3412$-avoiding involutions which was characterized by Barnabei et al.~ and the bijection between $(321,3\bar{1}42)$-avoiding permutations and Motzkin paths, presented by Chen et al.~. Furthermore, we manipulate the obtained continued fractions to get the recursion formulas for the polynomial distributions of crossings and nestings, and it follows that the results involve two new $q$-Motzkin numbers.
		\begin{center}
			
		\end{center}
	\end{abstract}

	\section{Introduction and result}\label{sec0}
		
		We let $S_n$ denote the set of all permutations of $[n]:=\{1,\ldots,n\}$. A permutation $\sigma\in S_n$ is an \textit{involution} whenever $\sigma(\sigma(i)) = i$ for all $i\in [n]$, and we write $I_n$ to denote the set of all involutions in $S_n$. We refer $|\sigma|$ as the length of the permutation $\sigma$.
		
		Suppose that $\sigma \in S_n$ and $\tau\in S_k$. We say a subsequence $s=\sigma(i_1)\sigma(i_2)\cdots \sigma(i_k)$ is an \textit{occurrence }of $\tau$ if and only if $s$ and $\tau$ are in the same isomorphic order, i.e. $\sigma(i_x)<\sigma(i_y)$ if and only if $\tau(x)<\tau(y)$. Whenever $\sigma$ contains no occurrence of $\tau$, we say that $\sigma$ \textit{avoids} the pattern $\tau$ or simply $\sigma$ is $\tau$-\textit{avoiding}. 	For any given set of permutations $T$, called set of patterns, we write $S_n(T)$
		and $I_n(T)$ to denote the sets of permutations and involutions, respectively, in $S_n$ which avoid every pattern in $T$, and we write $S(T)$ and $I(T)$ to denote the sets of all permutations and involutions, respectively, including the empty permutation, which avoid every pattern in $T$. Usually, if $T=\{\tau_1,\ldots,\tau_k\}$, we write $S_n(T)=S_n(\tau_1,\ldots,\tau_k)$ and $S(T)=S(\tau_1,\ldots,\tau_k)$. For example, the subsequence $5297$ of the permutation $\pi=35142987$ is an occurrence of $2143$. We can easily verify that $\pi\in I_9(4321)$.
		
		A \textit{barred permutation} is a permutation such that some of its elements are barred. We let $\bar{S}_k$ denote the set of all barred permutations of length $k$. Example: $3\bar{2}4\bar{5}1$ is a barred permutation in $\bar{S}_5$. Whenever $\bar{\tau}\in \bar{S}_k$, we let $\tau$ denote the permutation obtained by unbarring $\bar{\tau}$, and $\tau'$ denote the reduction of the permutation obtained from $\bar{\tau}$ by removing the barred element (if any). Example: if $\bar{\tau}=3\bar{2}4\bar{5}1$, we have $\tau=32451$ and $\tau'=231$.
		We say that a permutation $\sigma$ avoids $\bar{\tau}$ if each occurrence of $\tau'$ in $\sigma$ (if any) is part of an occurrence  of $\tau$ in $\sigma$. For example, a permutation $\sigma$ is $3\bar{1}42$-avoiding if and only if, for any occurrence $\sigma(i)\sigma(j)\sigma(k)$ of $321$ in $\sigma$, there exists $i <l< j$ such that $\sigma(l)<\sigma(k)$.
		
		A \textit{statistic} over a given set $E$ is a map $s:E\rightarrow \mathbb{N}$. The polynomial distribution of the statistic $s$ over the $E$ is the polynomial  $	\sum_{e \in E}q^{s(e)}$.	For example, if we let $\mathcal{P}_n$ denote the set of all subsets of $[n]$ and  ${\rm card}(A)$ denote the cardinality of $A$ for any $A\in \mathcal{P}_n$, then ${\rm card}$ is a statistic over  $\mathcal{P}_n$ and its polynomial distribution is
		\begin{align*}
		\sum_{A\in \mathcal{P}_n}q^{{\rm card}(A)}&=(1+q)^n.
		\end{align*}
		Furthermore, the generating function of the polynomial $\sum_{A\in \mathcal{P}_n}q^{{\rm card}(A)}$ is 
		\begin{align*}
			1+\sum_{n\geq 1}\sum_{A\in \mathcal{P}_n}q^{{\rm card}(A)}t^n&=\frac{1}{1-(1+q)t}.
		\end{align*}
		We will now define some statistics over $S_n$. For that, we let $\sigma \in S_n$. An \textit{excedance} (resp. \textit{fixed point}) of $\sigma$ is an index $i$ such that $\sigma(i)>i$ (resp. $\sigma(i)=i$). A \textit{crossing} (resp. \textit{nesting}, \textit{inversion})  of $\sigma$ is a pair of indices $(i,j)$ such that $i<j< \sigma(i)<\sigma(j)$  or  $\sigma(i)<\sigma(j)\leq i<j$ (resp. $i<j< \sigma(j)<\sigma(i)$  or  $\sigma(j)<\sigma(i)\leq i<j$, $i<j$ and $\sigma(j)>\sigma(i)$). We let $\exc(\sigma)$ (resp. $\crs(\sigma)$, $\nes(\sigma)$, $\inv(\sigma)$) denote the number of excedances  (resp. crossings, nestings, inversions) of $\sigma$. For example, the permutation $\pi=4\ 6\ 2\ 9\ 8\ 1\ 7\ 3\ 10\ 5\in S_{10}$ has the following properties: $\exc(\pi)=5$, $\crs(\pi)=7$, $\nes(\pi)=4$ and $\inv(\pi)=20$.  Therefore, $\exc$, $\crs$, $\nes$, and $\inv$ 	are all statistics over $S_n$. M\'edicis and Viennot \cite{MedVienot} and Randrianarivony \cite{ARandr} showed that these statistics are related by the following identity:
		\begin{align}\label{eq01}
		\inv(\sigma)=\exc(\sigma)+\crs(\sigma)+2\nes(\sigma), \text{ for any $\sigma \in S_n$}.
		\end{align} 
		
		In this paper, we are interested in the distribution of the statistics numbers of crossings and nestings over the sets $I_n(4321)$, $I_n(3412)$ and $S_n(321,3\bar{1}42)$.	These sets are all enumerated by the $n$-th Motzkin number (the sequence A001006 in \cite{OEIS}) and they are well studied in the literature (see \cite{Barnabei,Chen2,Egge} and references therein).
		
		Motzkin  numbers  $M_n$ are  traditionally  defined  by  the  following  recurrence  relation:
		\begin{align}\label{motz}
		M_0 = 1 \text{ and }  M_n = M_{n-1} +  \sum_{k=0}^{n-2}  M_{k}M_{n-2-k} \text{ for }  n \geq 1. 
		\end{align}
		The first values for $M_n$ are $1, 1, 2, 4, 9, 21, 51, 127, 323, 835, \ldots$~. Using \eqref{motz}, the  generating  function  for Motzkin  numbers, i.e.  $M(t)  =  \sum_{n\geq 0}M_nt^n$,  satisfies the following equivalent functional equations
		\begin{align}\label{motzk2}
		M(t) = 1 +  tM(t)  +  t^2M(t)^2 \text{ and } M(t)=\frac{1}{1-t-t^2M(t)}.
		\end{align}
		Solving the first equation of \eqref{motzk2} for $M(t)$, we obtain
		\begin{align*}
		M(t)&=\frac{1-t-\sqrt{1-2t-3t^2}}{2t^2}.
		\end{align*} 
		The second identity of \eqref{motzk2} leads to the following continued fraction expansion for $M(t)$
		\begin{align*}
		M(t)=\displaystyle \frac{1}{1-t-\displaystyle\frac{t^2}{1-t-\displaystyle\frac{t^2}{1-t-\displaystyle\frac{t^2}{\ \ \ \ \ddots}}}}.
		\end{align*}
		
		The term $q$-Motzkin numbers refers to a sequence of polynomials in $q$ whose evaluation at $q=1$ gives the sequence of Motzkin numbers. In other words, this polynomial satisfies the recurrence
	 for the Motzkin numbers and equals to the Motzkin numbers when specializing $q=1$.  Some $q$-generalizations of Motzkin numbers are studied by Barcucci et al.~\cite{Barcu} and recently by Barnabei and al.  \cite{Barnabei2}. Here, we are interested in two new $q$-Motzkin numbers $M_n(q)$ and $\tilde{M}_n(q)$. The first one $M_n(q)$ is defined by
		\begin{align}\label{M1}
		M_0(q)=1, M_n(q)=M_{n-1}(q)+\sum_{k=0}^{n-2}q^{k}M_{k}(q)M_{n-2-k}(q) \text{ for }n\geq 1.
		\end{align}
		and the second one $\tilde{M}_n(q)$ by
		\begin{align}\label{M2}
		\tilde{M}_0(q)=1,  \tilde{M}_n(q)=\tilde{M}_{n-1}(q)+\sum_{k=0}^{n-2}q^{k+1-(n-1)\delta_{k,n-2}}\tilde{M}_k(q)\tilde{M}_{n-2-k}(q) \text{ for }n\geq 1,
		\end{align}
		where $\delta$ is the usual Kronecker symbol. With a view to extend the results on $q$-Motzkin numbers introduced in \cite{Barcu,Barnabei2}, we will show how these $q$-Motzkin numbers are the distributions of crossings and nestings over restricted permutations, using continued fractions. For continued fractions, we are inspired by various known applications on enumeration of combinatorial objects (eg. \cite{Eliz3, Eliz4,Flajolet, ARandr1,ARandr}).	
			
		The study of the crossings and nestings over permutations was introduced by M\'edicis and Viennot in \cite{MedVienot} and  extended in \cite{BMP,Cort,Cort2,ARandr1,ARandr}. The study of these statistics over pattern-avoiding permutations was introduced by the second author \cite{Rakot} who exploited a known bijection of Elizalde and Pak \cite{ElizP} to find the equidistributions of the crossings over the sets of permutations avoiding the patterns 132, 213, and 321. Later, Rakotomamonjy et al.~\cite{Rakot2} enumerated the sets of permutations avoiding some pairs of patterns of length 3 according to $\crs$, and they found new combinatorial interpretations of some known sequences in \cite{OEIS}.
		In this work, we investigate the enumeration of some Motzkin objects, namely the sets $I_n(4321)$, $I_n(3412)$, and $S_n(321,3\bar{1}42)$, according to the numbers of crossings and nestings (see Theorem \ref{main}). The use of some known bijections between these sets and Motzkin paths allows us to reach our goal. 
		
		A \textit{Motzkin path} of length $n$ is a lattice path starting at $(0,0)$, ending at $(n,0)$, and never going below
		the $x$-axis, consisting of up steps $u = (1,1)$, horizontal steps $h = (1,0)$, and down steps $d = (1,-1)$. The set of Motzkin paths of length $n$ will be denoted by $\mathcal{M}_n$.

		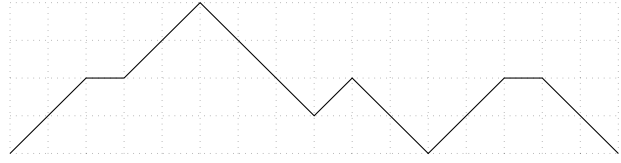
\begin{figure}[h]
			\begin{center}
				\begin{tikzpicture}
				\draw[step=0.5cm, gray, very thin,dotted] (0, 0) grid (8,2);
				%\draw [-, black](4, -0.2)--(4, 1.5);
				\draw[black] (0, 0)--(0.5, 0.5)--(1, 1)--(1.5, 1)--(2, 1.5)--(2.5, 2)--(3, 1.5)--(3.5, 1)--(4, 0.5)--(4.5, 1)--(5, 0.5)--(5.5, 0)--(6, 0.5)--(6.5, 1)--(7, 1)--(7.5, 0.5)--(8, 0);	
				\end{tikzpicture}
				\caption{The Motzkin path $P=uuhuudddudduuhdd\in \M_{16}$.}\label{fig0}
			\end{center}
		\end{figure}
		
		It is well known that the cardinality of $\mathcal{M}_n$ is the $n$-th Motzkin
		number $M_n$. The bijections between our combinatorial objects and Motzkin paths are well studied in \cite{Barnabei,Chen2}. Firstly, Barnabei et al.~\cite{Barnabei} studied the bijection of Biane \cite{Biane} restricted to involutions. The bijection of Biane maps an involution into a Motzkin path whose down steps are labelled with an integer that does not exceed their height, while
		the other steps are unlabelled. Barnabei et al.~ proved the following characterisations:
		\begin{itemize}
			\item An involution $\tau$ avoids $4321$ if and only if the label of any down step
			in the corresponding Motzkin path is $1$. 
			\item An involution $\tau$ avoids $3412$ if and only if the label of any down step in the corresponding Motzkin path equals its height.
		\end{itemize}
		Consequently, when restricted to $4321$- and $3412$-avoiding involutions, we can ignore the labels on Motzkin paths. Secondly, using reduced decomposition of Motzkin paths, Chen et al~\cite{Chen2} exhibited a bijection between $S_n(321,3\bar{1}42)$ and $\M_n$ which exchanges the statistics $\inv$ and $\area-\sh_{u}$. Through these two bijections, we interpret the statistics $\crs$ and $\nes$ in terms of statistics on Motzkin paths. Thus, as main results, we find the following combinatorial interpretations of $M_n(q)$ and $\tilde{M}_n(q)$.
		\begin{theorem} \label{main} We have the following identities
			\begin{align*}
			M_n(q)&=\sum_{\sigma\in I_n(4321)}q^{(\crs+\nes)(\sigma)}=\sum_{\sigma\in I_n(3412)}q^{\nes(\sigma)},\\
			\tilde{M}_n(q)&=\sum_{\sigma\in S_n(321,3\bar{1}42)}q^{\crs(\sigma)}.
			\end{align*}
		\end{theorem}
		We organise the rest of this paper in three sections. Section  \ref{sec2} is a preliminary section in which we define some statistics over Motzkin paths and we compute the generating function of their joint distribution. In Section \ref{sec3} and \ref{sec4}, making use of the result obtained in Section \ref{sec2} using the bijections of Biane \cite{Biane} and  Chen et al. \cite{Chen2}, we compute and express, in terms of continued fractions, the distributions of $\crs$ and $\nes$ over the sets $I_n(4321)$, $I_n(3412)$, and $S_n(321,3\bar{1}42)$. Using the obtained continued fractions, we provide the proof of Theorem \ref{main}.
		
		\section{Motzkin paths and statistics}\label{sec2}
		
		In this section, we define some statistics over Motzkin paths and we prove a preliminary result that is fundamental for the rest of the paper (see Theorem \ref{thm21}).
		
		Let $P=P_1P_2\cdots P_n$ be a Motzkin path of length $n$. Let us first define the following sets:
		\begin{align*}
		H\!or(P):=\{i|P_i=h\}, U\!p(P):=\{i|P_i=u\} \text{ and } Down(P):=\{i|P_i=d\}.
		\end{align*}
		The \textit{height} of each step  $s$ in $P$ is the $y$-coordinate of the starting point of $s$. Therefore, the height of the $i$-th step of $P$ is equal to 
		\begin{align*}
		h_i(P)=
		\begin{cases}
		|P(1,i)|_{u}-|P(1,i)|_{d}-1;& \text{ if } P_i=u\\
		|P(1,i)|_{u}-|P(1,i)|_{d};& \text{ if } P_i=h\\
		|P(1,i)|_{u}-|P(1,i)|_{d}+1;& \text{ if } P_i=d
		\end{cases},
		\end{align*} where $P(1,i):=P_1P_2\cdots P_i$ is the initial sub-path of length $i$ of $P$ and, for any word $w$, $|w|_a$ is the number of occurrences of the letter $a$ in the word $w$. Here are some statistics over $\M_n$: 
		\begin{itemize}
			\item $\area(P)$ the area between the path $P$ and the $x$-axis,
			\item $\up(P)$ (resp. $\hor(P)$,$\down(P)$) the number of up (resp. horizontal, down) steps of $P$,
			\item $\sh_{u}(P)$ (resp. $\sh_{h}(P)$, $\sh_{d}(P)$) the sum of heights of all up (resp. horizontal, down) steps of $P$, i.e.
			\begin{align*}
			\sh_{u}(P):=\sum_{i\in U\!p(P)}h_i(P), \sh_{h}(P):=\sum_{i\in Hor(P)}h_i(P), \text{ and } \sh_{d}(P):=\sum_{i\in Down(P)}h_i(P). 
			\end{align*}
		\end{itemize}
		For example, the Motzkin path in Fig. \ref{fig0} has the following properties: $\hor(P)=2$, $\up(P)=\down(P)=7$, $\sh_{u}(P)=8$, $\sh_{h}(P)=4$, $\sh_{d}(P)=15$ and $\area(P)=27$. We will show how these statistics are related.
		\begin{proposition}\label{prop21} For any Motzkin path $P$, we have 
			\begin{align*}
			\area(P)=2\sh_{d}(P)+\sh_{h}(P)-\down(P).
			\end{align*}	
		\end{proposition}
		\begin{proof}
			The proof is simply based on the following properties.
			\begin{itemize}
				\item If the $i$-th step of $P$ is a down step of height $h$, the area of the polygon $(i-1,h)-(i,h-1)-(i+1-h,0)-(i-1-h,0)$ is $2h-1$.
				\item If the $i$-th step of $P$ is a horizontal step of height $h$, the area of the polygon $(i-1,h)-(i,h)-(i+1-h,0)-(i-h,0)$ is $h$.
			\end{itemize} 
			The area of the path $P$ is equal to the sum of the areas of such polygons corresponding to down and horizontal steps (see  Figure \ref{fig1x}).
			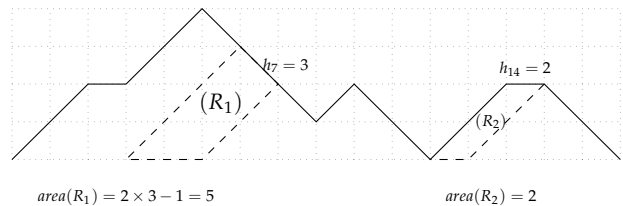
\begin{figure}[h]
				\begin{center}
					\begin{tikzpicture}
					\draw[step=0.5cm, gray, very thin,dotted] (0, 0) grid (8,2);
					%\draw [-, black](4, -0.2)--(4, 1.5);
					\draw[black] (0, 0)--(0.5, 0.5)--(1, 1)--(1.5, 1)--(2, 1.5)--(2.5, 2)--(3, 1.5)--(3.5, 1)--(4, 0.5)--(4.5, 1)--(5, 0.5)--(5.5, 0)--(6, 0.5)--(6.5, 1)--(7, 1)--(7.5, 0.5)--(8, 0);	
					\draw[-,dashed] (3, 1.5)--(3.5, 1)--(2.5, 0)--(1.5, 0)--(3, 1.5);
					%\draw[-,dashed] (4, 1)--(4.5, 1)--(5, 0.5)--(5.5, 0)--(6, 0.5);
					\draw[-,dashed] (7, 1)--(6, 0)--(5.5, 0);
					\draw[black] (1.5,-0.5) node[scale=0.5pt] {$area(R_1)=2\times 3-1=5$};
					\draw[black] (3.6, 1.25) node[scale=0.5pt] {$h_7=3$};
					\draw[black] (2.75, 0.75) node[scale=0.7pt] {$(R_1)$};
					\draw[black] (6.75, 1.2) node[scale=0.5pt] {$h_{14}=2$};
					\draw[black] (6.3, 0.5) node[scale=0.5pt] {$(R_2)$};
					\draw[black] (6.3,-0.5) node[scale=0.5pt] {$area(R_2)=2$};
					\end{tikzpicture}
					\caption{Decomposition of the area of a Motzkin path.}\label{fig1x}
				\end{center}
			\end{figure}
		\end{proof}
		By the same way, when we decompose the area statistic according to up and horizontal steps, we obtain the following proposition.
		\begin{proposition}\label{prop22} For any Motzkin path $P$, we have 
			\begin{align}\label{eq11y}
			\area(P)=2\sh_{u}(P)+\sh_{h}(P)+\up(P).
			\end{align}	
		\end{proposition}
		\begin{corollary}\label{cor21}
			For any Motzkin path $P$, we have $\sh_u(P)=\sh_{d}(P)-\down(P)$.
		\end{corollary}
		\begin{proof}
			Using both propositions \ref{prop21} and \ref{prop22} with the obvious fact that $\up(P)=\down(P)$ for any Motzkin path $P$, we obtain the corollary.
		\end{proof}
		We now consider the following polynomial	
			\begin{align*}
			I_{n}(a,b,c,d)=\sum_{P\in \mathcal{M}_n} a^{\hor(P)}b^{\up(P)}c^{\sh_{u}(P)}d^{\sh_{h}(P)},
			\end{align*}	
		and we denote by $\displaystyle I(a,b,c,d;t)=\sum_{n\geq 0}I_n(a,b,c,d)t^n$ the generating function of $I_n(a,b,c,d)$. Before closing this section, we will prove a fundamental identity which allows us to get easily the main results of this paper.
		 
		\begin{theorem}\label{thm21}
			The continued fraction expansion of $I(a,b,c,d;t)$ is
			\begin{align}\label{seq1eq1}
			\frac{1}{1-at-\displaystyle \frac{bt^2}{1-adt-\displaystyle \frac{bct^2}{1-ad^2t-\displaystyle \frac{bc^2t^2}{\ddots}}}}.
			\end{align}
		\end{theorem}
		
		\begin{proof}
			The proof is based on the usual decomposition of Motzkin paths. Let $P\in \M_n$. Using the first return decomposition, we obtain the following relations:
			\begin{itemize}
				\item if $P=hP'$ for some $P' \in \M_{n-1}$, then we have 
				\begin{align*}
				(\hor,\up,\sh_{u},\sh_{h})P=(1+\hor,\up,\sh_{u},\sh_{h})P'.
				\end{align*}
				\item If $P=uP_1dP_2$ for some $k\geq 2$ such that $P_1\in \M_{k-2}$ and $P_2\in \M_{n-k}$, then we have  
				\begin{align*}
				(\hor,\up,\sh_{u},\sh_{h})P=(\hor,1+\up,\sh_{u}+\up,\sh_{h}+\hor)P_1+(\hor,\up,\sh_{u},\sh_{h})P_2.
				\end{align*}
			\end{itemize}
			These two points lead to the following recurrence for the polynomial $I_{n}(a,b,c,d)$:
			\begin{align}\label{sec1eq2}
			I_n(a,b,c,d)=aI_{n-1}(a,b,c,d)+b\sum_{k=2}^{n}I_{k-2}(ad,bc,c,d)I_{n-k}(a,b,c,d),
			\end{align}
			with $I_0(a,b,c,d)=1$. Thus, when we compute $I(a,b,c,d;t)$, we obtain from \eqref{sec1eq2} the following relation
			\begin{align*}
			I(a,b,c,d;t)	&=\frac{1}{1-at-bt^2I(ad,bc,c,d;t)}.
			\end{align*}
			Developing $I(ad,bc,c,d;t)$ in its turn, we obtain the desired continued fraction expansion for $I(a,b,c,d;t)$.
		\end{proof}
		
		\section{Crossings and nestings over $I_n(4321)$ and $I_n(3412)$}\label{sec3}
		
		In this section, we will establish the proof of the first part of Theorem \ref{main} concerning the recursion for the polynomial distribution of crossings and nestings over the sets $I_n(4321)$ and $I_n(3412)$. For that, we first recall the bijection of Biane \cite{Biane} restricted to the set of involutions that was studied by Barnabei et al.~\cite{Barnabei}. The bijection maps an involution into a Motzkin path whose down steps are labelled with an integer that does not exceed its height, while
		the other steps are unlabelled. Barnabei et al.~\cite{Barnabei} proved the following characterisations:
		\begin{itemize}
			\item An involution $\sigma$ avoids $4321$ if and only if the label of any down step
			in the corresponding Motzkin path is $1$. 
			\item An involution $\sigma$ avoids $3412$ if and only if the label of any down step in the corresponding Motzkin path equals its height.
		\end{itemize}
		Consequently, when restricted to $4321$- and $3412$-avoiding involutions, we can ignore the labels on Motzkin paths. Let us denote by  $\Phi_1$ (resp. $\Phi_2$) the bijection from $\M_n$ to $I_n(4321)$ (resp. $I_n(3412)$).  In this section, we manipulate the bijections $\Phi_1$ and $\Phi_2$, and  we will show how these bijections interchange the statistics $(\fp,\exc,\crs,\nes)$  and  $(\hor,\up,a.\sh_{u},b.\sh_{h})$, where $a$ and $b$ are two arbitrary integers.  
		
		Since the construction of the image of an involution under the maps $\Phi_1$ and $\Phi_2$ is the same and obvious (see \cite{Barnabei}), we just focus on how to get the corresponding involutions from Motzkin paths. Let $P\in \M_{n}$. We will construct $\sigma_1=\Phi_1(P)$ by the following procedure (see Figure \ref{fig21} for graphical illustration).
		\begin{itemize}
			\item From left to right, number the steps of $P$ from $1$ up to $n$.
			\item Read the steps of $P$ from left to right and match the  $k$-th up step with  the $k$-th down step, $k=1,2,\ldots$.
			\item Then, we have
			\begin{itemize}
				\item[{\rm (a)}] $(i)$ is a $1$-cycle of $\sigma_1$, i.e. $\sigma_1(i)=i$, if and only if the step numbered  $i$ is horizontal.
				\item[{\rm (b)}] $(i,j)$ is a $2$-cycle of $\sigma_1$, i.e. $\sigma_1(i)=j$ and  $\sigma_1(j)=i$, if and only if 
				the $i$-th and the $j$-th step are matched.  
			\end{itemize}
		\end{itemize}
		
		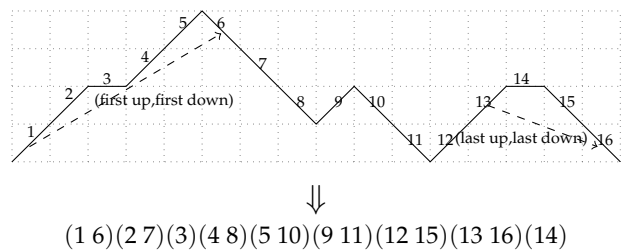
\begin{figure}[h]
			\begin{center}
				\begin{tikzpicture}
				\draw[step=0.5cm, gray, very thin,dotted] (0, 0) grid (8,2);
				\draw[step=0.5cm, gray, very thin,dotted] (0, 0) grid (8,2);
				\draw[black] (0, 0)--(0.5, 0.5)--(1, 1)--(1.5, 1)--(2, 1.5)--(2.5, 2)--(3, 1.5)--(3.5, 1)--(4, 0.5)--(4.5, 1)--(5, 0.5)--(5.5, 0)--(6, 0.5)--(6.5, 1)--(7, 1)--(7.5, 0.5)--(8, 0);	
				\draw[black] (0.25,0.4) node[scale=0.5pt] {$1$};\draw[black] (0.75,0.9) node[scale=0.5pt] {$2$};
				\draw[black] (1.25,1.1) node[scale=0.5pt] {$3$};\draw[black] (1.75,1.4) node[scale=0.5pt] {$4$};
				\draw[black] (2.25,1.85) node[scale=0.5pt] {$5$};\draw[black] (2.75,1.85) node[scale=0.5pt] {$6$};
				\draw[black] (3.3,1.3) node[scale=0.5pt] {$7$};\draw[black] (3.8,0.8) node[scale=0.5pt] {$8$};
				\draw[black] (4.3,0.8) node[scale=0.5pt] {$9$};\draw[black] (4.8,0.8) node[scale=0.5pt] {$10$};
				\draw[black] (5.3,0.3) node[scale=0.5pt] {$11$}; \draw[black] (6.7,1.1) node[scale=0.5pt] {$14$};
				\draw[black] (6.2,0.8) node[scale=0.5pt] {$13$};\draw[black] (7.3,0.8) node[scale=0.5pt] {$15$};
				\draw[black] (5.7,0.3) node[scale=0.5pt] {$12$};\draw[black] (7.8,0.3) node[scale=0.5pt] {$16$};
				\draw[->,dashed] (0.23,0.2)--(2.75,1.7); \draw[black] (2,0.8) node[scale=0.5pt] {(first up,first down)};
				\draw[->,dashed] (6.25, 0.75)--(7.7,0.2); \draw[black] (6.7, 0.3) node[scale=0.5pt] {(last up,last down)};
				
				\draw[black] (4,-0.5) node[scale=1pt] {$\Downarrow$};
				
				\draw[black] (4,-1) node[scale=0.75pt] {$(1\ 6)(2\ 7)(3)(4\ 8)(5\ 10)(9\ 11)(12\ 15)(13\ 16)(14)$};
				\end{tikzpicture}
				\caption{The corresponding $4321$-avoiding involution from a Motzkin path.}\label{fig21}
			\end{center}
		\end{figure}

		Before giving the formulation of $\Phi_2$, we first recall what a tunnel is.  The notion of tunnels on Dyck paths (Dyck paths are Motzkin paths having no horizontal step) is introduced by Elizalde et al.~\cite{ElizD,ElizP} and Barnabey et al.~\cite{Barnabei2} extend the notion in terms of Motzkin paths.  A \textit{tunnel} of a Motzkin path $P$ is a horizontal segment between two distinct lattice points of $P$ that intersects $P$ only at these two points and lies always below $P$. Each tunnel $t$ starts from the starting point of an up step $u$ to the ending point of a down step $d$. We say that the tunnel $t$ matches $u$ and $d$. For example, the Motzkin path in Figure \ref{fig22} has $7$ tunnels drawn in right dashed arrows. In particular, the tunnel $t_3$ matches the up step numbered 4 and the down step numbered 7.  For any Motzkin path $P$, we can construct $\sigma_2=\Phi_2(P)$ by the following procedure (see Figure \ref{fig22}):
		\begin{itemize}
			\item From left to right, number the steps of $P$ from $1$ up to $n$.
			\item Then, we have
			\begin{itemize}
				\item[{\rm (a)}] $(i)$ is a $1$-cycle of $\sigma_2$ if and only if the step numbered  $i$ is horizontal.
				\item[{\rm (b)}] $(i,j)$ is a $2$-cycle of $\sigma_2$ if and only if the up step numbered $i$ is matched with the down step numbered $j$ by a tunnel.
			\end{itemize}
		\end{itemize}
		\begin{figure}[h]
			\begin{center}
				\begin{tikzpicture}
				\draw[step=0.5cm, gray, very thin,dotted] (0, 0) grid (8,2);
				\draw[black] (0, 0)--(0.5, 0.5)--(1, 1)--(1.5, 1)--(2, 1.5)--(2.5, 2)--(3, 1.5)--(3.5, 1)--(4, 0.5)--(4.5, 1)--(5, 0.5)--(5.5, 0)--(6, 0.5)--(6.5, 1)--(7, 1)--(7.5, 0.5)--(8, 0);	
				\draw[->,dashed] (0,0)--(5.5, 0);
				\draw[->,dashed] (0.5, 0.5)--(5, 0.5);
				\draw[->,dashed] (1.5, 1)--(3.5, 1);
				\draw[->,dashed] (2, 1.5)--(3, 1.5);
				\draw[->,dashed] (6, .5)--(7.5, .5);
				\draw[->,dashed] (5.5, 0)--(8, 0);
				
				\draw[black] (0.25,0.4) node[scale=0.5pt] {$1$};\draw[black] (0.75,0.9) node[scale=0.5pt] {$2$};
				\draw[black] (1.25,1.1) node[scale=0.5pt] {$3$};\draw[black] (1.75,1.4) node[scale=0.5pt] {$4$}; \draw[black] (2.5,1.15) node[scale=0.8pt] {$t_3$};
				\draw[black] (2.25,1.85) node[scale=0.5pt] {$5$};\draw[black] (2.75,1.85) node[scale=0.5pt] {$6$};
				\draw[black] (3.3,1.3) node[scale=0.5pt] {$7$};\draw[black] (3.8,0.8) node[scale=0.5pt] {$8$};
				\draw[black] (4.3,0.8) node[scale=0.5pt] {$9$};\draw[black] (4.8,0.8) node[scale=0.5pt] {$10$};
				\draw[black] (5.3,0.3) node[scale=0.5pt] {$11$}; \draw[black] (6.7,1.1) node[scale=0.5pt] {$14$};
				\draw[black] (6.2,0.8) node[scale=0.5pt] {$13$};\draw[black] (7.3,0.8) node[scale=0.5pt] {$15$};
				\draw[black] (5.7,0.3) node[scale=0.5pt] {$12$};\draw[black] (7.8,0.3) node[scale=0.5pt] {$16$};
				
				\draw[black] (4,-0.5) node[scale=1pt] {$\Downarrow$};
				
				\draw[black] (4,-1) node[scale=0.75pt] {$(1\ 11)(2\ 8)(3)(4\ 7)(5\ 6)(9\ 10)(12\ 16)(13\ 15)(14)$};
				\end{tikzpicture}
				\caption{The corresponding $3412$-avoiding involution from a Motzkin path.}\label{fig22}
			\end{center}
		\end{figure}
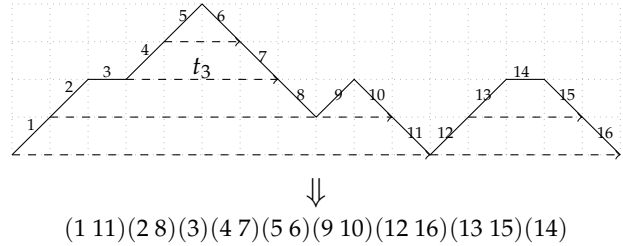
		
		By these definitions, it is not difficult to  show that the bijections $\Phi_1$ and $\Phi_2$ are well defined. 
		\begin{proposition}\label{prop31}
			Let $P$ be a Motzkin path. Assume that $\sigma_1=\Phi_1(P)$ and  $\sigma_2=\Phi_2(P)$. We have the following properties
			\begin{itemize}
				\item[{\rm (i)}] $(\crs,\nes)(\sigma_1)=(2\sh_{u},\sh_{h})(P)$.
				\item[{\rm (ii)}] $\nes(\sigma_2)=(2\sh_{u}+\sh_{h})(P)$.
			\end{itemize}
		\end{proposition}
		\begin{proof}
			Let $\sigma$ be an involution. By definition, we have
			\begin{align*}
			\nes(\sigma)&=2\sum_{\sigma(i)>i}|\{j<i|\sigma(j)>\sigma(i)\}| +\sum_{\sigma(i)=i}|\{j>i|\sigma(j)<i\}|,\\
			\text{ and } \crs(\sigma)&=2\sum_{\sigma(i)>i}|\{i<j<\sigma(i)|\sigma(j)>\sigma(i)\}|.
			\end{align*}
			Notice first that the following properties are obvious:
			\begin{itemize}
				\item[{\rm (a)}] if $\sigma$ is $4321$-avoiding, then we have $|\{j<i|\sigma(j)>\sigma(i)\}|=0$ for all $i$ such that $\sigma(i)>i$. 
				\item[{\rm (b)}]  if $\sigma$ is $3412$-avoiding, then we have $\crs(\sigma)=0$.
			\end{itemize}
			Let us suppose that $\sigma_1=\Phi_1(P)$ for some Motzkin path $P$. According to (a),  we have  
			\begin{align*}
			\nes(\sigma_1)&=\sum_{\sigma_1(i)=i}|\{j>i|\sigma_1(j)>i\}|\\
			&=\sum_{\sigma_1(i)=i}|\{k<i|\sigma_1(k)>i\}|\\
			&=\sum_{i\in Hor(P)} (|P(1,i)|_u-|P(1,i)|_d), \\
			&= \sh_{h}(P).
			\end{align*}
			Moreover, according to the given definition of $\Phi_1$, we have $|\{i<j<\sigma_1(i)|\sigma_1(j)>\sigma_1(i)\}|=|P(1,\sigma_1(i))|_{u}-|P(1,\sigma_1(i))|_{d}=h_{\sigma_1(i)}(P)-1$ (with $\sigma_1(i)$ as down step of $\sigma_1$). Consequently, we get
			\begin{align*}
			\crs(\sigma_1)&=2\sum_{\sigma_1(i)>i}|\{i<j<\sigma_1(i)|\sigma_1(j)>\sigma_1(i)\}|\\
			&=2\sum_{\sigma_1(i)>i}(h_{\sigma_1(i)}(P)-1)\\
			&=2\sum_{l\in Down(P)}(h_{l}(P)-1)\\
			&=2(\sh_{d}(P)-\down(P))\\
			&=2\sh_{u}(P) ~~~(\text{ see Corollary \ref{cor21}}).
			\end{align*}
			This ends the proof of (i).  To prove (ii), we now suppose that $\sigma_2=\Phi_2(P)$ for some Motzkin path $P$. For any excedance $i$ of $\sigma_2$, we have $|\{j<i|\sigma_2(j)>\sigma_2(i)\}|=|P(1,i))|_{u}-|P(1,i)|_{d}=h_i(P)$ and $i \in U\!p(P)$.
			\begin{align*}
			\nes(\sigma_2)&=2\sum_{\sigma_2(i)>i}|\{j<i|\sigma_2(j)>\sigma_2(i)\}| +\sum_{\sigma(i)=i}|\{j>i|\sigma_2(j)>i\}|,\\
			&=2\sum_{i\in U\!p(P)}h_i(P) +\sum_{i \in Hor(P)}h_i(P),\\
			&=2\sh_{u}(P)+\sh_{h}(P).
			\end{align*}
			This also ends the proof of (ii). Thus, Proposition \ref{prop31} follows.
		\end{proof}
		We now have all necessary tools to prove the following results concerning the joint distribution of the statistics $\fp$, $\exc$, $\crs$ and $\nes$ over our sets of 4321- and 3412-avoiding involutions.
		\begin{theorem} \label{thm31} We have the following identities
			\begin{align}\label{eqthm1}
			\sum_{\sigma \in I(4321)}x^{\fp(\sigma)}y^{\exc(\sigma)}p^{\crs(\sigma)}q^{\nes(\sigma)}t^{|\sigma|}&=\frac{1}{1-xt-\displaystyle \frac{yt^2}{1-xqt-\displaystyle \frac{yp^2t^2}{1-xq^2t-\displaystyle \frac{yp^4t^2}{\ddots}}}},
			\end{align}
			and
			\begin{align}\label{eqthm2}
			\sum_{\sigma\in I(3412)} x^{\fp(\sigma)}y^{\exc(\sigma)}q^{\nes(\sigma)}t^{|\sigma|}&=\frac{1}{1-xt-\displaystyle \frac{yt^2}{1-xqt-\displaystyle\frac{yq^2t^2}{1-xq^2t-\displaystyle\frac{yq^4t^2}{\ddots}}}}.
			\end{align}
		\end{theorem}
		\begin{proof} It is obvious to see from their definitions that the bijections $\Phi_1$ and $\Phi_2$  interchange the statistics $(\fp,\exc)$ and $(\hor,\up)$ (or  $(\hor,\down)$). So, using Proposition \ref{prop31}, we get
			\begin{itemize}
				\item $(\hor,\up,2\sh_u,\sh_h)\stackrel{\Phi_1}{\longrightarrow}(\fp,\exc,\crs,\nes) $
				\item and  $(\hor,\up,2\sh_u+\sh_h)\stackrel{\Phi_2}{\longrightarrow}(\fp,\exc,\nes)$
			\end{itemize}
			Consequently, we have
			\begin{align*}
			\sum_{\sigma \in I_n(4321)}x^{\fp(\sigma)}y^{\exc(\sigma)}p^{\crs(\sigma)}q^{\nes(\sigma)} &=I_n(x,y,p^2,q),\\
			\text{ and } \sum_{\sigma \in I_n(3412)}x^{\fp(\sigma)}y^{\exc(\sigma)}q^{\nes(\sigma)} &=I_n(x,y,q^2,q).
			\end{align*}
			So,  we obtain easily identity \eqref{eqthm1} (resp. \eqref{eqthm2}) from \eqref{seq1eq1},  by setting $a=x$, $b=y$, $c=p^2$ (resp. $c=q^2$) and $d=q$. This ends the proof of Theorem \ref{thm31}.
		\end{proof}
		
		Now, to close this section, we will prove the first identities of the main result presented in the introduction (see. Theorem \ref{main}).
		\begin{theorem}\label{thm33}
			We have $\displaystyle \sum_{\sigma \in I_n(4321)}q^{(\crs+\nes)(\sigma)}=\sum_{\sigma \in I_n(4312)}q^{\crs(\sigma)}=M_n(q)$.
		\end{theorem}
		\begin{proof}Let us denote by 
			\begin{align*}
			A(q;t)=\frac{1}{1-t-\displaystyle \frac{t^2}{1-qt-\displaystyle \frac{q^2t^2}{1-q^2t-\displaystyle \frac{q^4t^2}{\ddots}}}}.
			\end{align*}	
			Simple manipulation of the  continued fraction expansion of $A(q;t)$ leads to the following functional equation
			\begin{align*}
			A(q;t)=\frac{1}{1-t-t^2A(q;qt)}.
			\end{align*}
			This is also equivalent to the following one
			\begin{align*}
			A(q;t)=1+tA(q;t)+t^2A(q;t)A(q;qt).
			\end{align*}
			When we extract the coefficient $A_n(q)$ of $t^n$ on both sides of this equation, we obtain 
			\begin{equation*}
			A_0(q)=1, A_n(q) = A_{n-1}(q) +  \sum_{k=0}^{n-2}  q^kA_{k}(q)A_{n-2-k}(q)  \text{ for } n \geq 1. 
			\end{equation*}
			Since the polynomials $A_n(q)$ and $M_n(q)$ have the same recurrence (see recursion \eqref{motz}) and the same initial values, so they are the same. Furthermore, 	from \eqref{eqthm1} and \eqref{eqthm2}, we obtain
			\begin{align*}\label{eqthm3}
			\sum_{\sigma \in I(4321)}q^{(\crs+\nes)(\sigma)}t^{|\sigma|}=\sum_{\sigma \in I(3412)}q^{\crs(\sigma)}t^{|\sigma|}=A(q;t).
			\end{align*}
			Thus, we get Theorem \ref{thm33}.
		\end{proof}

		\section{Crossings over $S_n(321,3\bar{1}42)$}\label{sec4}
		
		The main goal of this section is to prove the second part of Theorem \ref{main} concerning the distribution of crossings over $S_n(321,3\bar{1}42)$. For that, we will use the bijection of Chen et al.~\cite{Chen2} based on reduced decomposition of permutations and strip decomposition of Motzkin paths.
		
		A \textit{transposition} $s_i =(i\  i+1)$ (i.e. $2$-cycle) is a map from $S_n$ to itself which interchanges the numbers in the $i$th and $(i+1)$th position in a permutation. For example, $s_3(623514) = 625314$. Every permutation $\sigma$ in $S_n$ can be represented as a sequence of transpositions $s_{i_1}s_{i_2}\cdots s_{i_k}$ which, when applied from right to left to the identity permutation $123\cdots n$, results in $\sigma$. Such representation is not necessarily unique. For example, $321$ can be written as both $s_1s_2s_1$ and $s_2s_1s_2$. In fact, the product of simple transpositions satisfies the Braid relations:
		\begin{align*}
		s_{i+1}s_is_{i+1} = s_is_{i+1}s_i \text{ and } s_is_j = s_js_i,\text{ where }  |i - j|\neq 1.
		\end{align*}
		Let $\sigma \in S_n$. We now determine an unique representation corresponding to $\sigma$, called the \textit{canonical reduced decomposition} of $\sigma$,  by a specific procedure. For that, we denote by $\mathcal{R}(\sigma)$ the set of the pairs resulting from the following procedure and initiate it to the empty set. 
		\begin{enumerate}
			\item If $\sigma$ is the identity, then the corresponding canonical reduced decomposition of $\sigma$ is the identity. Otherwise, we let $\sigma^{(0)}:=\sigma$ and we go to the next step.
			\item Locate the pair $(n_1,i_1)$, where $n_1$ is the greatest excedance value of $\sigma^{(0)}$  and $i_1$  its position (i.e. $i_1<n_1=\sigma^{(0)}(i_1)$ since $i_1$ is an excedance of $\sigma^{(0)}$).  Move $n_1$ from $i_1$ to position $n_1$ by applying from right to left the sequence $s_{n_1-1}s_{n_1-2}\cdots s_{i_1+1}s_{i_1}$ and leave the relative order of the other numbers unchanged.  Let $\sigma^{(1)}$ be the resulting permutation. Then, add the pair $(n_1-1,i_1)$ to $\mathcal{R}(\sigma)$ and pass to the next step.
			\item Return to the first step and apply the procedure to $\sigma^{(1)}$.	
		\end{enumerate}
		By this manner, if we have $\mathcal{R}(\sigma)=\{(h_1,t_1),(h_{2},t_{2}),\ldots,(h_k,t_k)\}$, i.e. $(h_{k+1-j},t_{k+1-j})=(n_j-1,i_j)$ for any $1\leq j\leq k$, then the  canonical reduced decomposition of $\sigma$ is $\sigma_1\sigma_2\cdots\sigma_k$, where $\sigma_{j}=s_{h_j}s_{h_j-1}\cdots s_{t_j}$ for any $1 \leq j\leq  k$. It is known that such decomposition is unique. We have to notice that the above algorithm is none other than that of Chen et al.\cite{Chen2}. We just modified it to emphasize the notion of excedance. In Chen et al.~\cite{Chen2}, the set $\mathcal{R}(\sigma)$ is known as the set of pairs (\textit{head,\textit{tail}}) and they showed the following characterisation.
		\begin{theorem} {\rm \cite[Thm. 2.3]{Chen2}}\label{chen1}  Let $\sigma \in S_n$ and $\mathcal{R}(\sigma)=\{(h_1,t_1),(h_{2},t_{2}),\ldots,(h_k,t_k)\}$. We have $\sigma \in S_n(321,3\bar{1}42)$ if and only if $t_{j-1}+2\leq t_{j}$ for $1< j\leq  k$.
		\end{theorem}
		Here are some examples, 
		\begin{itemize}
			\item the canonical reduced decomposition  of the identity $id=1 2\cdots n$ is the identity since $\mathcal{R}(id)=\emptyset$;
			\item the canonical reduced decomposition  of $\pi=n\ n-1\cdots 2 1$ is $\pi$ is $\pi= s_1/s_2s_1/$ $s_3s_2s_1/ \cdots /s_{n-1}s_{n-2}\cdots s_{1}$ since $\mathcal{R}(\pi)=\{(1,1),(2,1),\ldots,(n-1,1)\}$;
			\item if $\sigma=6~1~7~2~3~8~4~10~5~11~9~15~12~16~13~14$, we have \begin{align*}
			\mathcal{R}(\sigma)=\{(5,1),(6,3),(7,6),(9,8),(10,10),(14,12),(15,14)\}.
			\end{align*} 
			Thus, the canonical reduced decomposition of $\sigma$ is $\sigma_1\sigma_2\cdots\sigma_7$, with	$\sigma_1=s_{5}s_{4}s_{3}s_{2}s_{1}$, $\sigma_2=s_{6}s_{5}s_{4}s_{3}$, $\sigma_3=s_{7}s_{6}$, $\sigma_4=s_{9}s_{8}$, $\sigma_5=s_{10}$, $\sigma_6=s_{14}s_{13}s_{12}$, and $\sigma_7=s_{15}s_{14}$.
		\end{itemize}
		The following theorem is an extension of \cite[Thm. 2.4]{Chen2}.  
		\begin{theorem}{\rm \cite{Chen2}}\label{mazchen}  For any $\sigma \in S_n(321,3\bar{1}42)$, if   $\mathcal{R}(\sigma)=\{(h_r,t_r)\}_{1\leq r\leq k}$, then we have $Des(\sigma)=Exc(\sigma)=\{t_1, t_{2},\ldots,t_k\}$.
		\end{theorem}
		\begin{proof}
			Let $\sigma \in S_n(321,3\bar{1}42)$. Since  $\sigma$ is $321$-avoiding, then $\nes(\sigma)=0$ (see. \cite[Lem. 5.1]{Rakot}). Consequently, if $i_1<i_2<\cdots<i_k$ and  $j_1<j_2<\cdots<j_{n-k}$ are respectively excedances and non-excedances of $\sigma$, then we have $\sigma(i_1)<\sigma(i_2)<\cdots<\sigma(i_k)$ and $\sigma(j_1)<\sigma(j_2)<\cdots<\sigma(j_{n-k})$. 
			It is clear that, when we apply the above procedure, we get $\mathcal{R}(\sigma)=\{(\sigma(i_1)-1,i_1),(\sigma(i_2)-1,i_2),\ldots,(\sigma(i_k)-1,i_k)\}$. So, we have  $Exc(\sigma)=\{i_1, i_{2},\ldots,i_k\}$. For any $1\leq j\leq k$, since $i_j+2<i_{j+1}$ (see Theorem \ref{chen1}), then $i_j+1$ is a non-excedance of $\sigma$ and consequently we have $\sigma(i_j)>\sigma(i_j+1)$. Thus, excedances of $\sigma$ are all  descents of $\sigma$. Furthermore, for some $x$, a non-excedance $j_x$ of $\sigma$ can not be  descent of $\sigma$ because $j_x$ is followed by, either an excedance and we have $\sigma(j_x)<j_{x}+1<\sigma(j_{x}+1)$, or a non-excedance and we have $\sigma(j_{x})<\sigma(j_{x}+1)<j_{x}+1$. Thus, we get $Des(\sigma)=Exc(\sigma)$.
		\end{proof}
		
		Now, we recall the (x+y)-labelling and the strip decomposition of a Motzkin path.
		We call  the \textit{(x+y)-labelling} of a Motzkin path $P$ the action of labelling some cells (squares and triangulars containing a down-step) between $P$ and the $x$-axis by the sum of the coordinates of their left bottom coin.  We now define the strip decomposition of a Motzkin path. Suppose $P=P_{n,k}$ is a Motzkin path of length $n$ that contains $k$ up steps. If $k = 0$, then the strip decomposition of $P_{n,0}$ is simply the empty set. For any $P_{n,k} \in \M_n$, let $A \rightarrow B$ be the last up step and $E \rightarrow F$ the last down step on $P_{n,k}$. Then we define the strip of $P_{n,k}$ as the path from $B$ to $F$ along
		the path $P_{n,k}$. Now we move the points from $B$ to $E$ one layer lower, namely, subtract the $y$-coordinate by $1$, and denote the adjusted points by $B_0, \ldots, E_0$. We form a new Motzkin path by using the path $P_{n,k}$ up to the point $A$, then joining the point $A$ to $B_0$ and following the adjusted segment until we reach the point $E_0$, then continuing with the
		points on the $x$-axis to reach the destination $(n, 0)$. Denote this Motzkin path by $P_{n,k-1}$, which may end with some horizontal steps.
		From the strip of $P_{n,k}$, we may define the value $h_k$ as the label of the cell containing
		the step $E \rightarrow F$. Clearly, we have $h_k \leq n-1$. The value $t_k$ is defined as the label of the
		cell containing the step starting from the point $B$.
		
		Iterating the above procedure, we obtain the set $\mathcal{S}(P):=\{(h_i,t_i)\}_{1 \leq i \leq k}$, called the \textit{strip decomposition} of $P$, satisfying
		\begin{equation}\label{rel41}
		h_j<h_{j+1} \text{ and } t_j+2<t_{j+1} \text{ for any $1 \leq i < k$}.
		\end{equation}
		The strip decomposition of the Motzkin path $P=uuhduudddudduuhdd$ drawn in Figure \ref{fig41} is $\mathcal{S}(P)=\{(5,1),(6,3),(7,6),(9,8),(10,10),(14,12),(15,14)\}$. It is known that every Motzkin path $P$ can be determined from the set  $\mathcal{S}(P)$ satisfying \eqref{rel41} by reversing the above procedure ( see \cite[Fig. 2]{Chen2}). 
		
		\begin{figure}[h]
			\begin{center}
				\begin{tikzpicture}
				\draw[step=0.5cm, gray, very thin,dotted] (0, 0) grid (8,2);
				\draw[black] (0, 0)--(0.5, 0.5)--(1, 1)--(1.5, 1)--(2, 1.5)--(2.5, 2)--(3, 1.5)--(3.5, 1)--(4, 0.5)--(4.5, 1)--(5, 0.5)--(5.5, 0)--(6, 0.5)--(6.5, 1)--(7, 1)--(7.5, 0.5)--(8, 0);	
				
				\draw[-,dashed] (6, 0.5)--(7, 0.5)--(7.5, 0)--(8, 0);
				\draw[-,dashed] (4, 0.5)--(4.5, 0.5)--(5, 0)--(5.5, 0);
				\draw[-,dashed] (2, 1.5)--(2.5, 1.5)--(3.5, 0.5)--(4, 0)--(4.5, 0)--(5.5, 0);
				\draw[-,dashed] (1.5, 1)--(2.5, 1)--(3.5, 0)--(5, 0);
				%\draw[-,dashed] (1.5, 0.5)--(3, 0.5)--(3.5, 0)--(4, 0);
				\draw[-,dashed] (0.5, 0.5)--(2.5, 0.5)--(3, 0)--(3.5, 0);
				
				\draw[black] (7.7,0.2) node[scale=0.5pt] {$15$};\draw[black] (7.2,0.2) node[scale=0.5pt] {$14$}; \draw[black] (6.7,0.2) node[scale=0.5pt] {$13$}; 	\draw[black] (6.2,0.2) node[scale=0.5pt] {$12$}; 	\draw[black] (6.7,0.7) node[scale=0.5pt] {$14$};
				\draw[black] (7.2,0.7) node[scale=0.5pt] {$15$};
				
				\draw[black] (0.7,0.2) node[scale=0.5pt] {$1$};\draw[black] (1.2,0.2) node[scale=0.5pt] {$2$};\draw[black] (1.7,0.2) node[scale=0.5pt] {$3$};\draw[black] (2.2,0.2) node[scale=0.5pt] {$4$}; \draw[black] (2.7,0.2) node[scale=0.5pt] {$5$}; \draw[black] (3.2,0.2) node[scale=0.5pt] {$6$}; \draw[black] (3.7,0.2) node[scale=0.5pt] {$7$};\draw[black] (4.2,0.2) node[scale=0.5pt] {$8$}; \draw[black] (4.7,0.2) node[scale=0.5pt] {$9$};\draw[black] (5.2,0.2) node[scale=0.5pt] {$10$};
				
				\draw[black] (1.2,0.7) node[scale=0.5pt] {$3$};\draw[black] (1.7,0.7) node[scale=0.5pt] {$4$};\draw[black] (2.2,0.7) node[scale=0.5pt] {$5$}; \draw[black] (2.7,0.7) node[scale=0.5pt] {$6$}; \draw[black] (3.2,0.7) node[scale=0.5pt] {$7$}; \draw[black] (3.7,0.7) node[scale=0.5pt] {$8$}; \draw[black] (4.7,0.7) node[scale=0.5pt] {$10$};
				
				\draw[black] (2.2,1.2) node[scale=0.5pt] {$6$};\draw[black] (2.7,1.2) node[scale=0.5pt] {$7$}; \draw[black] (3.2,1.2) node[scale=0.5pt] {$8$};
				
				\draw[black] (2.7,1.7) node[scale=0.5pt] {$8$};
				\end{tikzpicture}
				\caption{The $(x + y)$-labelling and the strip decomposition of a Motzkin path.}\label{fig41}
			\end{center}
		\end{figure}
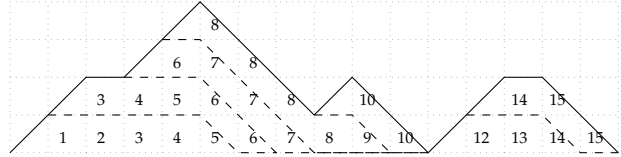
		
		According to Theorem \ref{chen1}, each strip decomposition of a Motzkin path $P$, i.e. $\mathcal{S}(P):=\{(h_i,t_i)\}_{1 \leq i \leq k}$, satisfying  \ref{rel41} is associated with a permutation $\sigma\in S_n(321,3\bar{1}42)$ throughout its reduced decomposition deduced from $\mathcal{R}(\sigma)=\{(h_i,t_i)\}_{1 \leq i \leq k}$, i.e. $\sigma_{1}\sigma_{2}\cdots \sigma_{k}$, where $\sigma_j=s_{h_j}s_{h_j-1}\cdots s_{t_j}$. In other words, the bijection $\Phi_3$ of Chen et al. from $\M_n$ to $S_n(321,3\bar{1}42)$ is defined as follows: for any $P\in \M_n$
		\begin{equation*}
		\sigma=\Phi_3(P) \text{ if and only if }  \mathcal{R}(\sigma)=\mathcal{S}(P).
		\end{equation*}

		\begin{proposition}
			Let $P\in \M_n$ and $\sigma=\Phi_3(P)$. We have 
			\begin{align*}
			(\exc,\crs)(\sigma)=(\up,\sh_{u}+\sh_{h})(P).
			\end{align*}
		\end{proposition}
		\begin{proof}
			Since the strip decomposition of $P$ is associated with up steps, we have 	$\up(P)=|\mathcal{S}(P)|=|\mathcal{R}(\sigma)|$. Moreover, using Theorem \ref{mazchen}, we obtain  $|\mathcal{R}(\sigma)|=\exc(\sigma)$. So, we get $\up(P)=\exc(\sigma)$. Chen et al. proved that we have $\inv(\sigma)=(\area-\sh_{u})(P)$ (see \cite[Thm. 3.2]{Chen2}). Finally, we get \begin{align*}
			\crs(\sigma)&=\inv(\sigma)-\exc(\sigma) \ \text{ using \eqref{eq01} with $\nes(\sigma)=0$}\\
			&=\area(P)-\sh_{u}(P)-\up(P)\\
			&=\sh_{u}(P)+\sh_{h}(P) \text{ using \eqref{eq11y} of Proposition \ref{prop22}}.
			\end{align*} 
			This ends the proof of our proposition.
		\end{proof}
		\begin{theorem}We have 
			\begin{align}
			\sum_{\sigma \in S(321,3\bar{1}42)}y^{\exc(\sigma)}q^{\crs(\sigma)}t^{|\sigma|}=\displaystyle \frac{1}{1-t-\displaystyle\frac{yt^2}{1-qt-\displaystyle\frac{yqt^2}{1-q^2t-\displaystyle\frac{yq^2t^2}{\ddots}}}}.
			\end{align}
		\end{theorem}
		\begin{proof} By the preceding proposition, we can get $\displaystyle\sum_{\sigma \in S_n(321,3\bar{1}42)} y^{\exc(\sigma)}q^{\crs(\sigma)}=I_n(1,y,q,q)$. 
			By setting $a=1$, $b=y$ and $c=d=q$ in \eqref{seq1eq1}, we obtain the theorem.
		\end{proof}
		We  now establish the proof of the second part of our main result concerning the recursion for the polynomials $\sum_{\sigma} q^{\crs(\sigma)}$, where the sum runs over all $\sigma  \in S_n(321,3\bar{1}42)$. For that, we recall the Stieltjes tableau $(h_{n,i})$ of two sequences $(\alpha_i)$ and $(\beta_i)$  considered by Dumont \cite{Dumont} in his master course on the $J$-continued fractions and Motzkin paths. He defined the tableau $(h_{n,i})$ through the following recurrence elsewhere
		\begin{align}\label{dumont}
		\begin{cases}
		h_{0,0}&=1, \text{ and } h_{n,i}=0  \text{ if } i<0 \text{ or } i>n; \\
		h_{n,i}&=\beta_ih_{n-1,i-1}+\alpha_{i+1}h_{n-1,i}+h_{n-1,i+1}, \text{ elsewhere}.
		\end{cases}
		\end{align}
		Using recursion \eqref{dumont},  Dumont showed the following theorem. To prove it, we can we can inquire into the paper of Dumont and Randrianarivony \cite{DumontArth}. 
		\begin{theorem}{\rm \cite{Dumont}} \label{dthm} We have
			\begin{align}
			\frac{1}{1-\alpha_1t-\displaystyle\frac{\beta_1t^2}{1-\alpha_2t-\displaystyle\frac{\beta_2t^2}{1-\alpha_3t-\displaystyle\frac{\beta_3t^2}{\ddots}}}}=h_{0,0}+h_{1,0}t+h_{2,0}t^2+\cdots + h_{n,0}t^n+\cdots~.
			\end{align}	
		\end{theorem}
		
	Here, we are interested in the particular Stieltjes tableau $(H_{n,i})$ obtained from  $(\alpha_i)$ and $(\beta_i)$ with $\alpha_{i}=\beta_{i}=q^{i-1}$ for any integer $i\geq 0$, i.e. 
		\begin{align}\label{sandr}
		\begin{cases}
		H_{0,0}&=1, \text{ and } H_{n,i}=0  \text{ if } i<0 \text{ or } i>n; \\
		H_{n,i}&=q^{i-1}H_{n-1,i-1}+q^iH_{n-1,i}+H_{n-1,i+1}, \text{ elsewhere}.
		\end{cases}
		\end{align}
		As direct corollary of Theorem \ref{dumont}, we have
		\begin{align}\label{eqthm4}
		\frac{1}{1-t-\displaystyle\frac{t^2}{1-qt-\displaystyle\frac{qt^2}{1-q^2t-\displaystyle\frac{q^2t^2}{\ddots}}}}=H_{0,0}+H_{1,0}t+H_{2,0}t^2+\cdots + H_{n,0}t^n+\cdots~.
		\end{align}
		\begin{table}[h!]
			\begin{center}
				\begin{tabular}{c|lllll}
					$n/k$& 0 & 1& 2 & 3 & 4 \\
					\hline
					0 & $1$ & & & & \\
					1 & $1$ & $1$ & & & \\
					2 & $2$ & $1+q$ & $q$ & &  \\
					3 & $3+q$ & $2+2q+q^2$ & $q+q^2+q^3$& $q^3$\\
					4 & $5+3q+q^2$ & $3\!+\!4q\!+\!3q^2\!+\!2q^3$ & $2q\!+\!2q^2\!+3q^3\!+\!q^4\!+\!q^5$ & $q^3\!+\!q^4\!+\!q^5\!+\!q^6$& $q^6$
				\end{tabular}
				\caption{Values of the $q$-tableau $(H_{n,i})$ for $0\leq n,i\leq 4$.}
				\label{tabhni}
			\end{center}
		\end{table}
		
		We will prove the following unexpected recursion satisfied by the sequence $(H_{n,i})$.
		\begin{theorem} \label{thm41} For any integers $n\geq 1$ and  $1\leq i\leq n$, we have
			\begin{align}\label{R0}
			H_{n,i}&=q^{i-1}(H_{n-1,i-1}+\sum_{k=i-1}^{n-2}q^{1+k}H_{k,i-1}H_{n-1-k,0}).
			\end{align}
		\end{theorem}
		\begin{proof} To prove \eqref{R0}, we will proceed by induction on $n$. For that, we can easily verify that \eqref{R0} holds for $n=1,2,3$. Suppose that it holds until some integer $n$. To prove that it still holds for $n+1$, we treat separately the cases $i=1$ and $i\geq 2$. By definition, we have
			\begin{align}\label{R1}
			H_{n+1,1}&= H_{n,0}+qH_{n,1}+H_{n,2}.
			\end{align}
			For $H_{n,2}$, we use \eqref{R0}  and the relationship $H_{k,1}=H_{k+1,0}-H_{k,0}$ for $k\geq 1$ to obtain
			\begin{align*}
			H_{n,2}&= q(H_{n-1,1}+\sum_{k=1}^{n-2}q^{1+k}H_{k,1}H_{n-1-k,0})\\
			&=q(H_{n,0}-H_{n-1,0}+\sum_{k=1}^{n-2}q^{1+k}(H_{k+1,0}-H_{k,0})H_{n-1-k,0})\\
			&=q(H_{n,0}+\sum_{k=2}^{n-1}q^{k}H_{k,0}H_{n-k,0}-(H_{n-1,0}+\sum_{k=1}^{n-2}q^{1+k}H_{k,0}H_{n-1-k,0})).
			\end{align*}
			Using the recurrence hypothesis for $H_{n,1}$, we obtain
			\begin{align*}
			H_{n,2}&=q(H_{n,0}+\sum_{k=2}^{n-1}q^{k}H_{k,0}H_{n-k,0}-(H_{n,1}-qH_{0,0}H_{n-1,0}))\\
			&=q(H_{0,0}H_{n,0}+qH_{1,0}H_{n-1,0}+\sum_{k=2}^{n-1}q^{k}H_{k,0}H_{n-k,0}-H_{n,1})  \text{ (since $H_{0,0}=H_{1,0}=1$)}\\
			&=\sum_{k=0}^{n-1}q^{1+k}H_{k,0}H_{n-k,0}-qH_{n,1}.
			\end{align*}	
			Thus, when returning to \eqref{R1}, we obtain
			\begin{align*}
			H_{n+1,1}	&=H_{n,0}+qH_{n,1}+\sum_{k=0}^{n-1}q^{1+k}H_{k,0}H_{n-k,0}-qH_{n,1}\\
			&=H_{n,0}+\sum_{k=0}^{n-1}q^{1+k}H_{k,0}H_{n-k,0}.	
			\end{align*}
			Now, we can treat by the same way the case $i\geq 2$ as follows. We start from the relation
			\begin{align}\label{R2}
			H_{n+1,i}= q^{i-1}H_{n,i-1}+q^iH_{n,i}+H_{n,i+1}.
			\end{align}
			Then, using \eqref{R0} for $H_{n,i+1}$, we get	
			\begin{align*}
			H_{n,i+1}&= q^{i-1}(H_{n-1,i}+\sum_{k=i}^{n-2}q^{1+k}H_{k,i}H_{n-1-k,0}).
			\end{align*}
			Now, replace $H_{k,i}$ by $H_{k+1,i-1}-q^{i-2}H_{k,i-2}-q^{i-1}H_{k,i-1}$ for any integer $i\geq 2$ and manipulate the obtained identity to  get
			\begin{align*}
			H_{n,i+1}&= q^{i}(H_{n,i-1}+\sum_{k+1}^{n-1}q^{k}H_{k,i-1}H_{n-k,0}-\underbrace{q^{i-2}(H_{n-1,i-2}+\sum_{k=i}^{n-2}q^{1+k}H_{k,i-2}H_{n-1-k,0})}_{A}\\
			&-\underbrace{q^{i}(H_{n-1,i-1}+\sum_{k=i}^{n-2}q^{1+k}H_{k,i-1}H_{n-1-k,0})}_{B}~)
			\end{align*}
			Using the recurrence hypothesis for $H_{n,i-1}$ and $H_{n,i}$, we get
			\begin{align*}
			A&=H_{n,i-1}-q^{i-2}(q^{i-1}H_{i-2,i-2}H_{n+1-i,0}+q^{i}H_{i-1,i-2}H_{n-i,0})\\
			&=H_{n,i-1}-q^{i-2}(H_{i-1,i-1}H_{n+1-i,0}+q^{i}H_{i-1,i-2}H_{n-i,0}) \text{ since $q^{i-1}H_{i-2,i-2}=H_{i-1,i-1}$}.
			\end{align*}
			and $B=H_{n,i}+q^{i-1}(q^{i}H_{i-1,i-1}H_{n-i,0})$. Thus, using also the fact that $q^{i-2}H_{i-1,i-2}+q^{i-1}H_{i-1,i-1}=H_{i,i-1}$, we get
			\begin{align*}
			A+B=H_{n,i}+H_{n,i-1}-q^{i-2}H_{i-1,i-1}H_{n+1-i,0}-q^{i}H_{i,i-1}H_{n-i,0}.
			\end{align*}
			Consequently, we obtain 
			\begin{align*}
			H_{n,i+1} 
			&=q^{i}(q^{i-2}H_{i-1,i-1}H_{n+1-i,0}+q^{i}H_{i,i-1}H_{n-i,0}+\sum_{k=i+1}^{n-1}q^{k}H_{k,i-1}H_{n-k,0}-H_{n,i})\\
			&=q^{i}(\sum_{k=i-1}^{n-1}q^{k}H_{k,i-1}H_{n-k,0}-H_{n,i})\\
			&=q^{i-1}\sum_{k=i-1}^{n-1}q^{1+k}H_{k,i-1}H_{n-k,0}-q^{i}H_{n,i}.
			\end{align*}
			So, when returning to \eqref{R2} and substitute $H_{n,i+1}$, we obtain		
			\begin{align*}
			H_{n+1,i}&=q^{i-1}(H_{n,i-1}+\sum_{k=i-1}^{n-1}q^{1+k}H_{k,i-1}H_{n-k,0}).
			\end{align*}
			This also ends the proof of Theorem \ref{thm41}.	
		\end{proof}
		\begin{corollary}\label{cor41}
			For any positive integer $n$, we have $H_{n,0}=\tilde{M}_n(q)$.
		\end{corollary}
		\begin{proof} Using \eqref{sandr} with the fact that $H_{n,0}=H_{n-1,0}+H_{n-1,1}$ for any integer $n\geq 0$, we get 
			\begin{align*}\label{R3}
			H_{n,0}&=H_{n-1,0}+H_{n-2,0}+\sum_{k=0}^{n-3}q^{1+k}H_{k,0}H_{n-1-k,0}\\
			&=H_{n-1,0}+\sum_{k=0}^{n-2}q^{1+k-(n-1)\delta_{k,n-2}}H_{k,0}H_{n-1-k,0}.
			\end{align*}
			Since $H_{n,0}$ and $\tilde{M}_n(q)$ satisfy the same recursion (see \eqref{M2}) and the same initial values, then they are equal.
		\end{proof}
		Combining this last corollary with identity \eqref{eqthm4}, we obtain our desired result.
		\begin{theorem} \label{thm42} For any positive integer $n$, we have
			\begin{equation*}
			\sum_{\sigma \in S_n(321,3\bar{1}42)}q^{\crs(\sigma)}=\tilde{M}_n(q).
			\end{equation*}
		\end{theorem}
		We close this section with an interesting non-trivial identity concerning equality of two continued fractions as another consequence of Corollary \ref{cor41}.
		\begin{theorem} \label{main12} We have the following identity
			\begin{align*}
			\frac{1}{1-\displaystyle\frac{t+t^2}{1-\displaystyle\frac{qt^2}{1-\displaystyle\frac{qt+(qt)^2}{1-\displaystyle\frac{q^3t^2}{\ddots}}}}}=\frac{1}{1-t-\displaystyle\frac{t^2}{1-qt-\displaystyle\frac{qt^2}{1-q^2t-\displaystyle\frac{q^2t^2}{\ddots}}}}.
			\end{align*}
		\end{theorem}
		\begin{proof} Since the polynomial $\tilde{M}_n(q)$ satisfies the following recursion 
			\begin{align*}
			\tilde{M}_0(q)=1, \tilde{M}_n(q)=\tilde{M}_{n-1}(q)+\tilde{M}_{n-2}(q)+\sum_{k=0}^{n-3}q^{1+k}\tilde{M}_k(q)\tilde{M}_{n-2-k}(q) \text{ for } n\geq 1,
			\end{align*}
			then its generating function, $\tilde{M}(q;t):=\sum_{n\geq 0}\tilde{M}_n(q)t^n$, satisfies	
			\begin{align*}
			\tilde{M}(q;t)&=\frac{1}{1-\displaystyle \frac{t+t^2}{1-qt^2\tilde{M}(q;qt)}}.
			\end{align*}
			So, when we continue to develop $\tilde{M}(q;q^kt)$ for $k\geq 1$, we obtain the continued fraction expansion for $\tilde{M}(q;t)$ 
			\begin{align*}
			\tilde{M}(q;t)=\frac{1}{1-\displaystyle\frac{t+t^2}{1-\displaystyle\frac{qt^2}{1-\displaystyle\frac{qt+(qt)^2}{1-\displaystyle\frac{q^3t^2}{\ddots}}}}}.
			\end{align*}
			Combining Corollary \ref{cor41} with identity \eqref{eqthm4}, we also have 	
			\begin{align*}
			\tilde{M}(q;t)=\frac{1}{1-t-\displaystyle\frac{t^2}{1-qt-\displaystyle\frac{qt^2}{1-q^2t-\displaystyle\frac{q^2t^2}{\ddots}}}}.
			\end{align*}	
			This proves the desired identity of our theorem.
		\end{proof}
	
	%%%%%%%%%%%%%%%%%%%%%%%%%%%%%%%%%%%%%%%%%%%%%%%%%
	\section*{Acknowledgments}
		We would like to thank Arthur Randrianarivony who never accepted that we state Theorem \ref{thm42} as a conjecture. A helpful discussion with him allows us to find the recursion \eqref{R0} of Theorem \ref{thm41}.
	
	%BIBLIOGRAPHY
	% You do not have to use the same format for your references, but 
	%    include everything in this file.  Don't use natbib please.
	% If you use BibTeX to create a bibliography, copy the .bbl file into here.
	% \newblock is optional (it adds a little space)

	\end{document}